\documentclass[11pt]{amsart}
\usepackage{amsmath}
\usepackage{amsfonts}
\usepackage{amssymb}
\usepackage{amsthm}
\usepackage{epsfig}
\usepackage{graphicx}
\usepackage{url, hyperref}

\newtheorem{theorem}{Theorem}[section]
\newtheorem{lemma}[theorem]{Lemma}
\newtheorem{corollary}[theorem]{Corollary}
\newtheorem{remark}[theorem]{Remark}

\newtheorem{definition}[theorem]{Definition}

\DeclareMathOperator{\inner}{int}
\DeclareMathOperator{\conv}{conv}
\DeclareMathOperator{\aff}{aff}

\DeclareMathOperator{\vol}{vol}

\def\R{\mathbb{R}}

\def\Z{\mathbb{Z}}
\def\Q{\mathbb{Q}}
\def\G{\mathrm{G}}
\def\Grat{\mathrm{Q}}
\def\d{\mathrm{d}}
\def\ind{\d}
\def\rd{\mathrm{q}}
\def\rind{\rd}
\newcommand{\mybinom}[2]{\left(\genfrac{}{}{0pt}{0}{#1}{#2}\right)}
\newcommand{\gauss}[1]{\left\lfloor #1\right\rfloor}
\newcommand{\lcm}[1]{\mathrm{lcm}\left(#1\right)}
\newcommand{\rat}[1]{\left\{#1\right\}}

\numberwithin{equation}{section}

\title{Rational Ehrhart quasi-polynomials}
\author{Eva Linke}
\address{Eva Linke, Fakult\"at f\"ur Mathematik, Universit\"at Mag\-deburg, Universit\"atsplatz 2, D-39106-Magdeburg,
Germany} 
\thanks{Supported by the Deutsche Forschungsgemeinschaft within the
project He 2272/4-1.}
\begin{document}

\begin{abstract}%{{{
Ehrhart's famous theorem states
that the number of integral points in a rational polytope is a quasi-polynomial in the integral dilation
factor. We study the case of rational dilation factors. It turns out that the number of integral points
can still be written as a rational quasi-polynomial. Furthermore, the coefficients of this rational 
quasi-polynomial are piecewise polynomial functions and related to each other by derivation.%}}}
\end{abstract}

\maketitle
\section{Introduction}%{{{
Let $\R^n$ be the $n$-dimensional Euclidean space and let $\Z^n$ be the integral lattice. 
For a set $M\subset\R^n$, we denote by $\inner(M)$ its interior, by $\conv(M)$ its convex hull, by
$\aff(M)$ its affine hull, by $\vol(M)$ its volume, which is the usual Lebesgue measure of $M$, and by
$\dim(M)$ its dimension, which is defined as the dimension of its affine hull. 
By $\vol_{\dim(M)}(M)$ we denote the $\dim(M)$-dimensional volume of $M$.
A polytope
is called integral (rational), if all its vertices have integral (rational) coordinates. For a rational polytope $P$, 
we denote by $\d(P)$ the denominator of $P$, that is, 
the smallest number $k\in\Z_{>0}$ such that $kP$ is an integral polytope. In other words
$\d(P)$ is the lowest common multiple of the denominators of all coordinates of the vertices of $P$.
Furthermore, let the $i$-index $\ind_i(P)$ of a rational polytope $P$ be the smallest number $k\in\Z_{\geq 0}$ 
such that for each $i$-face $F$ of $P$ the affine space $k\aff(F)$ contains integral points.\\
A function $p:\Z\to\Z$ is called a \emph{quasi-polynomial with period $d$} if there exist periodic functions 
$p_i:\Z_{\geq 0}\to\Z$ with period $d$ such that $p(k)=\sum_{i=0}^np_i(k)k^i$.
Ehrhart's Theorem states the following:

\begin{theorem}[Ehrhart, 1962, \cite{Ehrhart1962}, McMullen, 1978, \cite{McMullen1978}]\label{ehr_quasipoly}%{{{
  Let $P\subset\R^n$ be a rational polytope. Then
  \[\G(P,k):=\#(kP\cap\Z^n)=\sum_{i=0}^{\dim(P)}\G_i(P,k)k^i\ \text{ for }k\in\Z_{\geq 0}.\]
  $\G(P,\cdot):\Z_{\geq 0}\to\Z_{\geq 0}$ is called the \emph{Ehrhart quasi-polynomial} of $P$.
  For every $i$, the coefficient $\G_i(P,k)$ depends only on the congruence class of $k$ modulo $\ind_i(P)$.
	%}}}
\end{theorem}

Here, $\G_i(P,\cdot):\Z_{\geq 0}\to\Q$ is a periodic function and $\ind_i(P)$ is a period of $\G_i(P,\cdot)$.
The second part of this statement is due to McMullen. Furthermore, $\G_0(P,0)=1$, and $\G_{\dim(P)}(P,k)=\vol_{\dim(P)}(P)$ for all 
$k\in\Z_{\geq 0}$ such that $k\aff(P)$ contains integer points.  

For an introduction into Ehrhart theory we refer to
Beck and Robins \cite{Beck&Robins2006}.
Unfortunately, $\ind_i(P)$ is not necessarily the minimal period of $\G_i(P,\cdot)$, that is, there is
possibly a smaller integer number $p<\ind_i(P)$ such that $p$ is a period of $\G_i(P,\cdot)$. 
This phenomenon is called period collapse and has been subject to active research in the last years.
McAllister and Woods \cite{McAllister&Woods} studied the $1$- and $2$-dimensional case with the result that period collapse
does not occur in dimension $1$, and they gave a characterization of those rational polygons in dimension $2$
whose Ehrhart quasi-polynomial is a polynomial. They also showed that the minimal periods of $\G_i(P,\cdot)$ are not
necessarily decreasing with $i$.
Woods \cite{Woods2004} gave a polynomial-time algorithm in fixed dimension
which decides whether a given integer $p$ is a period of all $\G_i(P,\cdot)$. In \cite{Beck&Sam&Woods}, Beck, Sam and Woods constructed
polytopes with no period collapse at all. Furthermore, they showed that period collapse never occurs for $\G_{\dim(P)-1}(P,\cdot)$.
Haase and McAllister \cite{Haase&McAllister} gave a conjectural explanation of period collapse involving splitting the polytope into
pieces and applying unimodular transformations onto these pieces.

We show that the Ehrhart quasi-polynomial can be generalized to a rational quasi-polynomial by allowing rational dilation factors,
where a \emph{rational quasi-polynomial with period $d$} is a function $p:\Q\to\Q$ of the 
form $p(r)=\sum_{i=0}^np_i(r)r^i$ where $p_i:\Q\to\Q$ is periodic with period $d$.

\begin{theorem}\label{thm:rat_dilations}%{{{
  Let $P\subset\R^n$ be a rational polytope. Then for $r\in\Q_{\geq 0}$
  \[\Grat(P,r):=\#(rP\cap\Z^n)=\sum_{i=0}^{\dim(P)}\Grat_i(P,r)r^i.\]
  Here, $\Grat_i(P,\cdot):\Q_{\geq 0}\to\Q$ is a periodic function, and $\ind_i(P)$ is a period of $\Grat_i(P,\cdot)$.
	We call $\Grat(P,\cdot):\Q_{\geq 0}\to\Q$ the \emph{rational Ehrhart quasi-polynomial} of $P$. 
	%}}}
\end{theorem}

We remark that $\Grat_i(P,\cdot)$ is an extension of $\G_i(P,\cdot)$ to rational numbers. Hence,
$\Grat_0(P,0)=\G_0(P,0)=1$. Furthermore, we have that $\Grat_{\dim(P)}(P,r)=\vol_{\dim(P)}(P)$ for all $r$ such that
$r\aff(P)$ contains integer points.

We define a rational analogue of the $i$-index. These rational indices allow us to show a refined result on the periods of rational Ehrhart quasi-polynomials.

\begin{definition}%{{{
Let the \emph{rational denominator} $\rd(P)$ of $P$ be the smallest positive rational number $r$ 
such that $rP$ is an integral polytope:
\begin{equation*}
\rd(P)=\min\{r\in\Q_{>0}: rP\text{ is an integral polytope}\},
\end{equation*}
and let the \emph{rational $i$-index} $\rind_i(P)$ of $P$ 
be the smallest positive rational number $r$ such that for each $i$-face $F$ of $P$ the affine space $r\aff(F)$ 
contains integral points:
\begin{equation*}
\rind_i(P)=\min\{r\in\Q_{>0}: r\,\aff(F)\cap\Z^n\neq\emptyset\emph{, for all }i\emph{-faces } F\}.
\end{equation*}
%}}}
\end{definition}

Then we get the following result:

\begin{corollary}\label{cor:periods}%{{{
Let $P$ be a rational polytope. Then $\rind_i(P)$ is a period of $\Grat_i(P,\cdot)$.
Furthermore, $\Grat(P,\cdot):\Q_{\geq 0}\to\Z$ is a rational quasi-polynomial with period $\rind_0(P)=\rd(P)$.
%}}}
\end{corollary}

Ehrhart's reciprocity law is also true for rational Ehrhart quasi-polynomials:

\begin{corollary}\label{cor:reci}%{{{
Let $P$ be a rational polytope and let $\Grat(P,r)=\sum_{i=0}^{\dim(P)}\Grat_i(P,r)r^i$ be its rational Ehrhart quasi-polynomial.
Then for $r\in\Q_{\geq 0}$,
\begin{equation*}
\#(r\,\inner(P)\cap\Z^n)=(-1)^{\dim(P)}\Grat(P,-r).
\end{equation*}
%}}}
\end{corollary}

We show further that the coefficients $\Grat_i(P,\cdot)$ of the rational Ehrhart quasi-polynomial are piecewise polynomials.
Here, we assume $P$ to be full-di\-men\-sional. This assumption is necessary since, if $P$ is contained in an affine hyper\-plane
not containing $0$, we have $\#(rP\cap\Z^m)=0$ for all $r\in\Q_{\geq 0}$ such that $r\aff(P)$ does not contain integral points. Thus, in that case $\Grat_i(P,r)=0$
for nearly all points $r\in\Q_{\geq 0}$. On the other hand, if $P$ is contained in a hyperplane containing $0$, it behaves like 
a full-dimensional polytope. 

\begin{theorem}\label{thm:polycoeffs}%{{{
Let $P\subset\R^n$ be an $n$-dimensional rational polytope and let $\Grat(P,r)=\sum_{i=0}^n\Grat_i(P,r)r^i$ be its rational Ehrhart
quasi-polynomial. Then $\Grat_i(P,\cdot)$ is a piecewise polynomial of degree $n-i$.
%}}}
\end{theorem}

By $\Grat_i'(P,r)$ we denote the first derivative of $\Grat_i(P,r)$ in $r$ if it exists. Using this
we deduce the following relation between the coefficients of the rational Ehrhart quasi-polynomial of
a polytope $P$:

\begin{theorem}\label{thm:derivatives}%{{{
Let $P\subset\R^n$ be an $n$-dimensional rational polytope. Then 
\begin{equation*}
\Grat_i'(P,r)=-(i+1)\Grat_{i+1}(P,r),\quad i=0,\ldots,n-1,
\end{equation*} 
for all $r\geq 0$ where the derivative exists.%}}}
\end{theorem}
This theorem can be seen as a first step towards investigations on period collapses; it implies that,
in contrast to the integral case, the minimal periods of $\Grat_i(P,r)$ are decreasing with $i$
if $0\in P$.

In general nothing is known about extremal values of $G_i(P,k)$ or $\Grat_i(P,r)$.  
As a first result in this direction we have in the 2-dimensional case:

\begin{theorem}\label{thm:dim2}%{{{
Let $P$ be an arbitrary $2$-dimensional rational polygon. Then $|\Grat_1(P,r)|\leq \Grat_1(P,0)$ for all $0\leq r<\rd(P)$.
%}}}
\end{theorem}
To this end, we work out an explicit example using the approach presented by Sam and Woods in \cite{Sam&Woods2009}.
An analogous statement for $\Grat_0(P,r)$ is not true. 

The paper is organized as follows: In Section \ref{sec:rat} we present all tools used for the proofs of
the results presented in this introduction. The proofs of Theorems \ref{thm:rat_dilations}, \ref{thm:polycoeffs} and \ref{thm:derivatives}
and their Corollaries are given in this section as well. In Section \ref{sec:dim2} we work out a detailed example
(see Theorem~\ref{thm_2d_triangle})
and deduce Theorem \ref{thm:dim2} from the explicit formulas of the rational Ehrhart quasi-polynomial of this example.

For further information on Ehrhart (quasi-)polynomials and similar problems as considered in this work, we refer to 
\cite{Barvinok1992,Beck&Haase2008,Brion&Vergne,Chen&Li&Sam2010}.
%}}}
%%%%%%%%%%%%%%%%%%%%%%%%%%%%%%%%%%%%%%%%%%%%%%%%%%%%%%%%%%%%%%%%%%%%%%%%%%%%%%%%%%%%%%%%%%%%%%%%%%%%%%%%%%%%%%
\section{Rational dilations}\label{sec:rat}%{{{
To prove Theorem \ref{thm:rat_dilations}, we need a well-known property of the $\G_i(\cdot,\cdot)$, which we present here:

\begin{lemma}\label{lem:int_dilations}%{{{
Let $\G(P,k)=\sum_{i=0}^{\dim(P)}\G_i(P,k)k^i$ be the Ehrhart quasi-polynomial of a rational polytope $P$. Then
$\G_i(mP,k)=\G_i(P,mk)m^i$ for $m,k\in\Z_{\geq 0}$.
%}}}
\end{lemma}
\begin{proof}%{{{
We have $\#(mkP\cap\Z^n)=\G(mP,k)=\sum_{i=0}^{\dim(P)}\G_i(mP,k)k^i$ and 
$\#(mkP\cap\Z^n)=\G(P,mk)=\sum_{i=0}^{\dim(P)}\G_i(P,mk)(mk)^i$.
Comparing coefficients yields $\G_i(mP,k)=\G_i(P,mk)m^i$, $\forall m,k\in\Z_{\geq 0}$ (see Barvinok 
\cite[Section 4.3]{Barvinok2006} for details on equality of quasi-polynomials).
%}}}
\end{proof}
\begin{proof}[Proof of Theorem \ref{thm:rat_dilations}]%{{{
Let $\G(P,k)=\sum_{i=0}^{\dim(P)}\G_i(P,k)k^i$ for $k\in\Z_{\geq 0}$ be the Ehrhart quasi-polynomial of $P$. 
We define 
\begin{equation*}
\Grat_i\left(P,\frac{a}{b}\right):=\G_i\left(\frac{1}{b}P,a\right)b^i.
\end{equation*}
$\Grat_i\left(P,\frac{a}{b}\right)$ is well-defined, since for $\frac{a}{b}=\frac{ka}{kb}$
we get $\Grat_i\left(P,\frac{ka}{kb}\right)=\G_i(\frac{1}{kb}P,ka)k^ib^i=\G_i(\frac{1}{b}P,a)b^i=\Grat_i\left(P,\frac{a}{b}\right)$ by Lemma \ref{lem:int_dilations}.
Then
\begin{equation*}
\Grat\left(P,\frac{a}{b}\right)
=\G\left(\frac{1}{b}P,a\right)
=\sum_{i=0}^{\dim(P)}\G_i\left(\frac{1}{b}P,a\right)a^i=\sum_{i=0}^{\dim(P)}\Grat_i\left(P,\frac{a}{b}\right)\left(\frac{a}{b}\right)^i.
\end{equation*}
It remains to show that $\Grat_i\left(P,\frac{a}{b}\right)$ is periodic with period $\ind_i(P)$.  
Since $b\,\ind_i(P)$ is a multiple of $\ind_i\left(\frac{1}{b}P\right)$, we get
\begin{equation*}
\Grat_i\left(P,\frac{a}{b}+\ind_i(P)\right)
=\G_i\left(\frac{1}{b}P,a+b\,\ind_i(P)\right)b^i\\
=\G_i\left(\frac{1}{b}P,a\right)b^i
=\Grat_i\left(P,\frac{a}{b}\right).\qedhere
\end{equation*}
%}}}
\end{proof}

This proof implies that knowing the classical Ehrhart quasi-polynomial of $\frac{1}{b}P$
for all positive integers $b$ is equivalent to knowing the rational Ehrhart quasi-polynomial of $P$.
However, as the next remark shows, it is not enough to know the Ehrhart quasi-polynomial of a polytope
to recover the rational version:

\begin{remark}%{{{
$\Grat(P,\cdot):\Q_{\geq 0}\to\Z$ is not invariant under translations of $P$ with respect to integral
vectors. Furthermore, $\Grat(P,\cdot):\Q_{\geq 0}\to\Z$ is not necessarily monotonically increasing if $0\not\in P$.
For instance, let
\begin{equation*}\begin{split}
T_1&=\conv\left(\binom{1/2}{-1/2},\binom{-1/2}{-1/2},\binom{0}{3/2}\right),\\
T_2&=\conv\left(\binom{1/2}{1/2},\binom{-1/2}{1/2},\binom{0}{5/2}\right)
\end{split}\end{equation*}
(see Figure \ref{fig:rem}).Then $T_2=T_1+\binom{0}{1}$.
Nevertheless, we have $\Grat(T_1,2/3)=2$ and $\Grat(T_2,2/3)=1$. Moreover, $\Grat(T_2,2)=7$ and $\Grat(T_2,11/5)=4$.
\begin{figure}[ht]\centering
\includegraphics{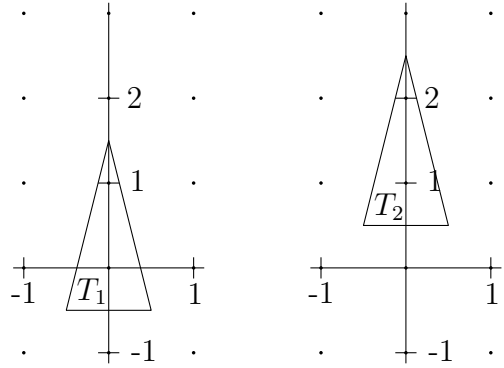}
\caption{Triangles $T_1$ and $T_2$}\label{fig:rem}
\end{figure}
%}}}
\end{remark}

Furthermore, as for the integral case, there are examples of polytopes with different combinatorial type
that have the same rational Ehrhart quasi-polynomial. Stanley \cite{STANLEY1986} constructed two classes of
polytopes, \emph{order polytopes} and \emph{chain polytopes}. In general these are polytopes of different combinatorial type 
but with the same Ehrhart polynomials. Since his consideration is independent of the integrality of dilation factors, these polytopes 
have also the same rational Ehrhart polynomials.

\begin{lemma}\label{lem:rat_dilations}%{{{
For all $r,s\in\Q$ we have $\Grat_i(sP,r)=\Grat_i(P,sr)s^i$.
%}}}
\end{lemma}
\begin{proof}%{{{
Let $r=\frac{a}{b}$, $s=\frac{c}{d}$. By the definition of $\Grat_i(P,r)$ in the proof of Theorem \ref{thm:rat_dilations},
we get, together with Lemma \ref{lem:int_dilations},
\begin{equation*}\begin{split}
\Grat_i(sP,r)&=\Grat_i\left(\frac{c}{d}P,\frac{a}{b}\right)=\G_i\left(\frac{c}{db}P,a\right)b^i\quad \text{and}\\
\Grat_i(P,sr)s^i&=\Grat_i\left(P,\frac{ac}{bd}\right)\frac{c^i}{d^i}
=\G_i\left(\frac{c}{bd}P,a\right)b^i.\qedhere
\end{split}\end{equation*}
%}}}
\end{proof}
As in the integral case, the rational indices are divisors of each other.
\begin{lemma}\label{lem:rat_divisors}%{{{
Let $P$ be a rational polytope. Then $\rind_{i-1}(P)/\rind_{i}(P)\in\Z$ for $i=1,\ldots,n$.
%}}}
\end{lemma}
\begin{proof}%{{{
Let $H^i_1,\ldots,H^i_{m(i)}$ be the respective affine hulls of the $m(i)$ $i$-faces of $P$ and let
$r^i_j$ be the smallest positive rational number such that $H^i_j$ contains integral points.
Then $rH^i_j$ contains integral points if and only if $r$ is an integral multiple of $r^i_j$, for $j=1,\ldots,m(j)$. Thus, 
$\rind_i(P)$ is the smallest positive rational number that is an integral multiple of all $r^i_j$.
Furthermore, since $H^{i-1}_j\subset H^{i}_{\tilde{\jmath}}$ for some $\tilde{\jmath}$, we have that $r^{i-1}_{j}$ is an
integral multiple of $r^{i}_{\tilde{\jmath}}$, and thus $\rind_{i-1}(P)$ is an integral multiple of $\rind_{i}(P)$.
%}}}
\end{proof}
Now we are able to prove that the rational indices are periods of the coefficients of the rational Ehrhart quasi-polynomials.
\begin{proof}[Proof of Corollary \ref{cor:periods}]%{{{
Since $\ind_i(\rind_i(P)P)=1$ for all $i$, we know that
\[\Grat_i(\rind_i(P)P,r+k)=\Grat_i(\rind_i(P)P,r),\quad\forall r\in\Q_{\geq_0},k\in\Z_{\geq_0}.\]
This implies, together with Lemma \ref{lem:rat_dilations},
\[\Grat_i(P,r\,\rind_i(P)+k\,\rind_i(P))=\Grat_i(P,r\,\rind_i(P)),\quad\forall r\in\Q_{\geq_0},k\in\Z_{\geq_0},\]
and thus
\[\Grat_i(P,\tilde{r}+k\,\rind_i(P))=\Grat_i(P,\tilde{r}),\quad\forall \tilde{r}\in\Q_{\geq_0},k\in\Z_{\geq_0}.\]
Furthermore, together with Lemma \ref{lem:rat_divisors}, we get that $\rind_0(P)/\rind_{i}(P)\in\Z$ for $i=1,\ldots,n$,
and thus $\rind_0(P)=\rd(P)$ is a period of $\Grat(P,\cdot)$.
%}}}S
\end{proof}

Concerning the minimal periods of Ehrhart quasi-polynomials, Beck, Sam and Woods \cite{Beck&Sam&Woods} showed that 
$\ind_{\dim(P)-1}(P)$ is in fact the minimal period of $\G_{\dim(P)-1}(P,\cdot)$. An analogous result is also true in the rational case:

\begin{corollary}%{{{
Let $P\subset\R^n$ be a rational polytope and let $\Grat(P,r)=\sum_{i=0}^{\dim(P)}\Grat_i(P,r)r^i$ be its rational Ehrhart quasi-polynomial.
Then $\rind_{\dim(P)-1}(P)$ is the minimal period of $\Grat_{\dim(P)-1}(P,\cdot)$.
%}}}
\end{corollary}
\begin{proof}%{{{
By Lemma \ref{lem:rat_dilations} and since $\rind_{\dim(P)-1}$ is homogeneous, it suffices to show the statement for all $P$
with $\rind_{\dim(P)-1}(P)=1$. Thus we assume that $\frac{s}{t}<1$ is a period of $\Grat_{\dim(P)-1}(P,\cdot)$
with $s,t\in\Z_{>0}$, that is $\Grat_{\dim(P)-1}(P,r)=\Grat_{\dim(P)-1}\left(P,r+\frac{s}{t}\right)$ for all $r\in\Q_{\geq 0}.$
Again by Lemma \ref{lem:rat_dilations} we get
\begin{equation*}
\Grat_{\dim(P)-1}\left(\frac{1}{t}P,rt\right)=\Grat_{\dim(P)-1}\left(\frac{1}{t}P,rt+s\right)\quad\text{for all }r\in\Q_{\geq 0}.
\end{equation*}
In particular this is true for all $rt\in\Z_{\geq 0}$, and
thus, $s$ is a period of $\G_{\dim(P)-1}\left(\frac{1}{t}P,\cdot\right)$, which is a contradiction, since
\begin{equation*}
\ind_{\dim(P)-1}\left(\frac{1}{t}P\right)\geq\rind_{\dim(P)-1}\left(\frac{1}{t}P\right)=t\,\rind_{\dim(P)-1}(P)=t>s.\qedhere
\end{equation*}
%}}}
\end{proof}

We can also prove the Ehrhart reciprocity law for rational Ehrhart quasi-polynomials. We refer to Beck and Robins \cite[Chapter 4]{Beck&Robins2006}
for details on Ehrhart reciprocity law.

\begin{proof}[Proof of Corollary~\ref{cor:reci}]%{{{
Let $r=\frac{a}{b}$ with $a,b\in\Z_{\geq 0}$. Then, by Ehrhart reciprocity law and Lemma \ref{lem:rat_dilations}, we get
\begin{equation*}\begin{split}
\#\left(\frac{a}{b}\,\inner(P)\cap\Z^n\right)
&=(-1)^{\dim(P)}\sum_{i=0}^{\dim(P)}\G_i\left(\frac{1}{b}P,-a\right)(-a)^i\\
&=(-1)^{\dim(P)}\sum_{i=0}^{\dim(P)}\Grat_i\left(P,-\frac{a}{b}\right)\left(\frac{1}{b}\right)^i(-a)^i\\
&=(-1)^{\dim(P)}\Grat\left(P,-\frac{a}{b}\right).\qedhere
\end{split}\end{equation*}
%}}}
\end{proof}

For the proof of Theorem \ref{thm:polycoeffs} we need the following Lemma:

\begin{lemma}\label{qp_const_lemma}%{{{
Let $p:\Q\to\Q$ be a rational quasi-polynomial of degree $n\in\Z_{>0}$ with period $d\in\Q_{>0}$ and constant leading coefficient, that is,
\[p(r)=p_nr^n+p_{n-1}(r)r^{n-1}+p_{n-2}(r)r^{n-2}+\ldots+p_1(r)r+p_0(r),\]
where $0\neq p_n\in\Q$ and $p_i:\Q\to\Q$ are periodic functions with period $d$ for $i=0,\ldots,n-1$. Furthermore, 
suppose there exist an interval $(r_1,r_2)\subset \Q$ and $c_k\in\Q$ for $k\in\Z_{\geq 0}$ such that
\[p(r+kd)=c_k,\quad\forall r\in (r_1,r_2), \forall k\in\Z_{\geq 0}.\]
Then $p_i:(r_1,r_2)\to\Q$ is a polynomial of degree $n-i$. Furthermore, if $p_n>0$ then $p_{n-1}$ has negative leading
coefficient.
%}}}
\end{lemma}
\begin{proof}%{{{
We prove this result by induction with respect to $n$. For $n=1$ we have $c_0=p(r)=p_1r+p_0(r)$ for all $r\in(r_1,r_2)$.
Thus, $p_0(r)=c_0-p_1r$ for $r\in(r_1,r_2)$, which is a polynomial of degree $n-0=1$ with negative leading coefficient.\\
Now let $n>1$. We have
\begin{equation*}
c_k=p_n\cdot(r+kd)^n+\sum_{i=0}^{n-1}p_i(r)(r+kd)^i\quad\forall r\in(r_1,r_2), \forall k\in\Z_{\geq 0}.
\end{equation*}
Then $q:\Q\to\Q$ with $q(r):=p_n\cdot\left((r+d)^n-r^n\right)+\sum_{i=0}^{n-1}p_i(r)\left((r+d)^i-r^i\right)$ is a
quasi-polynomial of degree $n-1$ with period $d$ and constant leading coefficient, and
\begin{equation*}\begin{split}
q(r+md)&=p_n\cdot\left((r+(m+1)d)^n-(r+md)^n\right)\\
&\hspace{1cm}+\sum_{i=0}^{n-1}p_i(r)\left((r+(m+1)d)^i-(r+md)^i\right)\\
&=c_{m+1}-c_m\quad\forall r\in(r_1,r_2), \forall m\in\Z_{\geq 0}.
\end{split}\end{equation*}
Thus, we can use the induction hypothesis for $q$, and together with
\begin{equation*}\begin{split}
q(r)&=p_n\cdot\left((r+d)^n-r^n\right)+\sum_{i=0}^{n-1}p_i(r)\left((r+d)^i-r^i\right)\\
&=p_nndr^{n-1} + \sum_{j=0 }^{n-2}\left(p_n\binom{n}{j}d^{n-j}+\sum_{i=j+1 }^{n-1}p_i(r)\binom{i}{j}d^{i-j}\right)r^j\\
\end{split}\end{equation*} 
we get that 
\begin{equation*}
q_j(r):=p_n\binom{n}{j}d^{n-j}+\sum_{i=j+1 }^{n-1}p_i(r)\binom{i}{j}d^{i-j}
\end{equation*}
is a polynomial of degree $n-1-j$
for $r\in(r_1,r_2)$, for all $j=0,\ldots,n-2$. Since $p_nnd>0$ we get, also by induction, that 
$q_{n-2}(r)=p_n\binom{n}{2}d^2+p_{n-1}(r)\binom{n-1}{n-2}d$ has a negative leading coefficient. \\
Now we use induction again to show 
that $p_{j+1}$ is a polynomial of degree $n-j-1$ 
for $r\in(r_1,r_2)$.
For $j=n-2$ we have that $q_{n-2}(r)=p_n\binom{n}{2}d^2+p_{n-1}(r)\binom{n-1}{n-2}d$ is a polynomial of degree $1$ with negative leading coefficient,
hence the same is true for $p_{n-1}$.
For $j<n-2$ write
\begin{equation*}
q_j(r)=p_n\binom{n}{j}d^{n-j}+\sum_{i=j+1 }^{n-1}p_i(r)\binom{i}{j}d^{i-j}=\sum_{i=0}^{n-j-1}\alpha_ir^i\quad \forall r\in(r_1,r_2). 
\end{equation*}
Then for $r\in(r_1,r_2)$,
\begin{equation*}
p_{j+1}(r)(j+1)d=\sum_{i=0}^{n-j-1}\alpha_ir^i-\sum_{i=j+2 }^{n-1}p_i(r)\binom{i}{j}d^{i-j}-p_n\binom{n}{j}d^{n-j}
\end{equation*}
which is a polynomial of degree $n-j-1$ since $p_i(r)$ is a polynomial of degree $n-i$ for $i\geq j+2$ by induction
hypothesis.\newline 
We conclude that $p_i(r)$ is a polynomial of degree $n-i$ for $r\in(r_1,r_2)$ and $i=1\ldots,n-1$. That $p_0(r)$ is
a polynomial follows immediately from $p_0(r)=c_0-p_nr^n-\sum_{i=1}^{n-1}p_i(r)r^i$.
%}}}
\end{proof}

Next we show that Ehrhart quasi-polynomials of full-dimensional rational polytopes satisfy the setting in Lemma \ref{qp_const_lemma}.
\begin{proof}[Proof of Theorem \ref{thm:polycoeffs}]%{{{
By Theorem \ref{thm:rat_dilations}, $\Grat(P,r)$ is a rational quasi-polynomial of degree $n$ with 
period $\rd(P)$ and constant, nonzero leading coefficient. To apply Lemma \ref{qp_const_lemma} it remains to show that there exist 
$0=r_0<r_1<\ldots<r_l=\rd(P)$ such that $\Grat(P,r)$ is constant for $r\in (r_i+k\rd(P),r_{i+1}+k\rd(P))$ and $i=0,\ldots,l-1$, $k\in\Z_{\geq 0}$.
To this end we consider $\Grat(P,r)$ as a function $\Q_{\geq 0}\to \Q$. $\Grat(P,r)$ is certainly piecewise constant, and it 
jumps whenever integral points leave or enter $rP$, which can only happen if one of the facets of $rP$ lies
in a hyperplane containing integral points. Thus, for every facet $F$ of $P$ let $\alpha_F$ be the smallest positive rational number
such that $\alpha_FF$ lies in a hyperplane containing integral points. Then $\{k\alpha_F: F\text{ facet of }P, k\in\Z_{\geq 0}\}$
are the only possible jump discontinuities of $\Grat(P,r)$. By the definition of $\rd(P)$, for a facet $F$ of $P$ there exists a $k_F\in\Z$
such that $k_F\alpha_F=\rd(P)$. Thus for $\{r_0,\ldots,r_l\}=\{k\alpha_F: k=0,\ldots,k_F, F\text{ facet of }P\}$ we can apply
Lemma~\ref{qp_const_lemma}. 
%}}}
\end{proof}

We refer to Figure \ref{g01pic} in Section \ref{sec:dim2} for a visualization of the $\Grat_i(P,\cdot)$.

\begin{remark}%{{{
The rational Ehrhart quasi-polynomials can be extended to real quasi-polynomials 
$\Grat(P,\cdot):\R_{\geq 0}\to\Z_{\geq 0}$.
To do that, for $r_0\in\R_{\geq 0}\setminus\Q$ let $r_j\in\Q_{\geq 0}$, $j\geq 1$ with $r_0=\lim_{j\to\infty}r_j\in\R$. We define 
$\Grat_i(P,r_0):=\lim_{j\to\infty}\Grat_i(P,r_j)$. This limit exists since $\Grat_i(P,\cdot)$ is piecewise continuous and,
for $j$ large enough, all $r_j$ are contained in the same continuous part of $\Grat_i(P,\cdot)$. Then, since $\Grat(P,\cdot):\R_{\geq 0}\to\Z_{\geq 0}$
only jumps for rational points, we get
\begin{equation*}\begin{split}
\#(r_0P\cap\Z^n)&=\lim_{j\to\infty}\#(r_jP\cap\Z^n)=\lim_{j\to\infty}\sum_{i=0}^n\Grat_i(P,r_j)r_j^i\\
&=\sum_{i=0}^n\lim_{j\to\infty}\Grat_i(P,r_j)r_j^i=\sum_{i=0}^n\Grat_i(P,r_0)r_0^i.
\end{split}\end{equation*}
Baldoni et al.~\cite{KOEPPEetAl2011} generalized this statement to intermediate sums of rational polytopes
and gave an efficient algorithm to compute these real quasi-polynomials.
%}}}
\end{remark}
Now we show that the coefficients $\Grat_i(P,\cdot)$ 
are derivatives of each other.
\begin{proof}[Proof of Theorem \ref{thm:derivatives}]%{{{
Let $0=r_0<r_1<\ldots<r_l=\rd(P)$ be as in the proof of Theorem \ref{thm:polycoeffs}, and for $m=1,\ldots,l$, $k\in\Z_{\geq 0}$ let
\[c_{m,k}=\Grat(P,r)=\sum_{i=0}^n\Grat_i(P,r)r^i\quad \text{for }r\in(r_{m-1}+k\,\rd(P),r_{m}+k\,\rd(P)).\] 
Since $\Grat_i(P,r)$ is a polynomial of degree $n-i$ in $r$, we can write it as
\begin{equation*}
\Grat_i(P,r)=\sum_{j=0}^{n-i}\Grat_{i,j}r^j.
\end{equation*}
Since $\Grat_i(P,r)$ are periodic with period $\rd(P)$, we can write $r=\tilde{r}+k\,\rd(P)$ with $k\in\Z_{\geq 0}$ and 
$r_{m-1}\leq \tilde{r}<r_{m}$ for some $m=1,\ldots,l$ and get
\begin{equation*}\begin{split}
c_{m,k}&=\sum_{i=0}^n\Grat_i(P,\tilde{r})(\tilde{r}+k\,\rd(P))^i
=\sum_{i=0}^n\sum_{j=0}^{n-i}\Grat_{i,j}\tilde{r}^j(\tilde{r}+k\,\rd(P))^i\\
&=\sum_{h=0}^n\sum_{i=0}^{n}\sum_{j=\max(0,h-i)}^{\min(h,n-i)}\binom{i}{h-j}\Grat_{i,j}(k\,\rd(P))^{i-h+j}\tilde{r}^{h},\\
\end{split}
\end{equation*}
which is a constant polynomial in $\tilde{r}$. Thus, for $h\neq 0$,
\begin{equation*}
\sum_{i=0}^{n}\sum_{j=\max(0,h-i)}^{\min(h,n-i)}\binom{i}{h-j}\Grat_{i,j}(k\,\rd(P))^{i-h+j}=0
\end{equation*}
and therefore
\begin{equation*}
\Grat(P,\tilde{r}+k\,\rd(P))
=\sum_{i=0}^{n}\Grat_{i,0}(k\,\rd(P))^{i}=\sum_{i=0}^{n}\Grat_{i,0}(r-\tilde{r})^{i}.
\end{equation*}
Expanding to the quasi-polynomial form yields
\begin{equation*}
\begin{split}
\Grat(P,\tilde{r}+k\,\rd(P))&=\sum_{i=0}^{n}\Grat_{i,0}\sum_{j=0}^i\binom{i}{j}r^j\tilde{r}^{i-j}(-1)^{i-j}\\
&=\sum_{j=0}^{n}\left(\sum_{i=0}^{n-j}\binom{i+j}{j}\Grat_{i+j,0}(-1)^{i}\tilde{r}^{i}\right)r^j.
\end{split}
\end{equation*}
This implies that for all $\tilde{r}\in (r_{m-1},r_m)$, $m=1,\ldots,l$,
\begin{equation*}
\Grat_j(P,\tilde{r})=\sum_{i=0}^{n-j}\binom{i+j}{j}\Grat_{i+j,0}(-1)^{i}\tilde{r}^{i}
\end{equation*}
and the claim follows by differentiation. 
%%}}}
\end{proof}
%}}}
%%%%%%%%%%%%%%%%%%%%%%%%%%%%%%%%%%%%%%%%%%%%%%%%%%%%%%%%%%%%%%%%%%%%%%%%%%%%%%%%%%%%%%%%%%%%%%%%%%%%%%%%%%%%%%%
\section{Dimension 2}\label{sec:dim2}%{{{
In what follows, we denote by $\gauss{.}$ the floor function, that is, $\gauss{x}$ is the largest integer not greater than $x$,
by $\lceil.\rceil$ the ceiling function, that is, $\lceil x\rceil$ is the smallest integer not smaller than $x$,
and by $\rat{.}$ the fractional part, that is, $\rat{x}=x-\gauss{x}$.
For the following calculations, we mention the following fact: Let $n,m,t,r\in\Z$, $m>0$ and $t\equiv r\bmod m$. Then
$\gauss{\frac{nt}{m}}=\frac{nt}{m}-\rat{\frac{nr}{m}}$ and
$\left\lceil\frac{nt}{m}\right\rceil=\frac{nt}{m}+\rat{-\frac{nr}{m}}$.

First, we consider $2$-dimensional triangles of the form
$T=\conv\left(\binom{0}{0},\binom{x_1}{y},\binom{x_2}{y}\right)$,
where $x_1<x_2\in\Q$ and $y\in\Q_{>0}$.
\begin{figure}[ht]\centering
\includegraphics{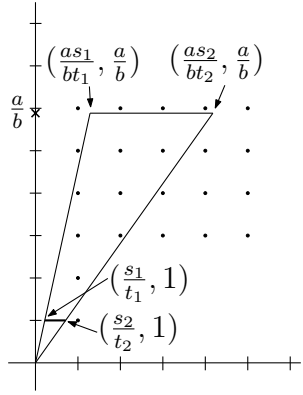}
\caption{Triangle $T$}\label{triangle2d}
\end{figure}
\begin{theorem}\label{thm_2d_triangle}%{{{
Let $T=\conv\left(\mybinom{0}{0},\mybinom{\frac{s_1}{t_1}\frac{a}{b}}{\frac{a}{b}},\mybinom{\frac{s_2}{t_2}\frac{a}{b}}{\frac{a}{b}}\right)$
 with $a,b,t_1,t_2\in\Z_{>0}$, $s_1,s_2\in\Z$, $\frac{s_2}{t_2}>\frac{s_1}{t_1}$, and $\gcd(a,b)=\gcd(s_1,t_1)=\gcd(s_2,t_2)=1$.
Then for $r\in\Q_{\geq 0}$ the following hold:
\begin{equation*}\begin{split}
\mathrm{(i)}\quad\Grat_2(T,r)&=\frac12\frac{a^2}{b^2}\left(\frac{s_2}{t_2}-\frac{s_1}{t_1}\right)\\
\mathrm{(ii)}\quad\Grat_1(T,r)&=\frac{a}{b}\left(\frac{t_1+t_2}{2t_1t_2}-\left(\rat{\frac{ar}{b}}-\frac12\right)\left(\frac{s_2}{t_2}-\frac{s_1}{t_1}\right)\right) \\
\mathrm{(iii)}\quad\Grat_0(T,r)&=1
-\frac{1}{2}\rat{\frac{ar}{b}}\left(\frac{s_2}{t_2}-\frac{s_1}{t_1}+2\right)
+\frac{1}{2}\rat{\frac{ar}{b}}^2\left(\frac{s_2}{t_2}-\frac{s_1}{t_1}\right)\\
&\quad+\rat{\frac{ar}{b\lcm{t_1,t_2}} }\lcm{t_1,t_2}\left(\frac{t_2-1}{2t_2}+\frac{t_1-1}{2t_1}\right)\\
&\quad-\sum_{i=0}^{\scriptscriptstyle\rat{\gauss{\frac{ar}{b}} /\lcm{t_1,t_2} }\lcm{t_1,t_2}}
\left(\frac{s_2i}{t_2}-\gauss{\frac{s_2i}{t_2}}+\left\lceil\frac{s_1i}{t_1}\right\rceil-\frac{s_1i}{t_1}\right).
\end{split}\end{equation*}
%}}}
\end{theorem}

\begin{proof}%{{{
In what follows, we determine $\Grat(T,t)=\sum_{i=0}^{\gauss{at/b}}\Grat(Q,i)$, $t\in\Q_{\geq 0}$,
$Q=\conv\left(\mybinom{\frac{s_1}{t_1}}{1},\mybinom{\frac{s_2}{t_2}}{1}\right)$, (see Figure \ref{triangle2d},
\cite{Sam&Woods2009}).
For abbreviation we write $l$ instead of $\lcm{t_1,t_2}$ and $r$ for an arbitrary integer number with $r\equiv t\mod lb$.
It is
\begin{equation*}\begin{split}
\Grat(Q,t)&=\#(tQ\cap\Z^2)=\gauss{\frac{s_2t}{t_2}}-\left\lceil\frac{s_1t}{t_1}\right\rceil+1\\
&=\left(\frac{s_2}{t_2}-\frac{s_1}{t_1}\right)t-\left(\rat{\frac{s_2r}{t_2}}+\rat{-\frac{s_1r}{t_1}}\right)+1.
\end{split}\end{equation*}  
This implies
\begin{equation}\label{eq:2d-theorem}
\Grat(T,t)
=\sum_{i=0}^{\gauss{at/b}}\left(\frac{s_2}{t_2}-\frac{s_1}{t_1}\right)i-\left(\rat{\frac{s_2i}{t_2}}
+\rat{1-\frac{s_1i}{t_1}}\right)+1.
\end{equation}
Since $\gauss{at/b}=\frac{at}{b}-\rat{\frac{ar}{b}}$, the first part can be written as
\begin{equation*}
\sum_{i=0}^{\gauss{at/b}}i=t^2\frac{a^2}{2b^2}
+t\frac{a}{b}\left(\frac12-\rat{\frac{ar}{b}}\right)
+\frac12\rat{\frac{ar}{b}}^2-\frac12\rat{\frac{ar}{b}}.
\end{equation*}
For the second part, we remark that $\rat{\frac{s_2i}{t_2}}$ is periodic with period $t_2$ and
\begin{equation*}
\sum_{i=0}^{l-1}\rat{\frac{s_2i}{t_2}}=\frac{l}{t_2}\sum_{i=0}^{t_2-1}\rat{\frac{s_2i}{t_2}}
=\frac{l(t_2-1)}{2t_2}.
\end{equation*}
Thus, we get
\begin{equation*}\begin{split}
\sum_{i=0}^{\gauss{at/b}}\rat{\frac{s_2i}{t_2}}
&=\gauss{\frac{ \gauss{\frac{at}{b}} } {l}}\sum_{i=0}^{l-1}\rat{\frac{s_2i}{t_2}}
+\sum_{i=0}^{\scriptscriptstyle\rat{\gauss{\frac{ar}{b}} /l }l}\rat{\frac{s_2i}{t_2}}\\
&=\left(\frac{at}{bl}-\rat{\frac{ar}{bl}}\right)\frac{l(t_2-1)}{2t_2}
+\sum_{i=0}^{\scriptscriptstyle\rat{\gauss{\frac{ar}{b}} /l }l}\rat{\frac{s_2i}{t_2}},
\end{split}\end{equation*}
and similarly
\begin{equation*}
\sum_{i=0}^{\gauss{at/b}}\rat{-\frac{s_1i}{t_1}}
=\left(\frac{at}{bl}-\rat{\frac{ar}{bl}}\right)\frac{l(t_1-1)}{2t_1}
+\sum_{i=0}^{\scriptscriptstyle\rat{\gauss{\frac{ar}{b}} /l }l}\rat{1-\frac{s_1i}{t_1}}.
\end{equation*}
After some elementary algebra, \eqref{eq:2d-theorem} expands to the claim. 
%}}}
\end{proof}

In particular, the theorem shows that (as shown in Section \ref{sec:rat}) $\frac{b}{a}$ is a period of $\Grat_1(T,\cdot)$,
$\Grat_1(T,\cdot)$ is piecewise linear,
and that $\frac{b\lcm{t_1,t_2}}{a}$ is a period of $\Grat_0(T,r)$. Furthermore, $\Grat_0(T,r)$ is piecewise quadratic, and 
the pieces are equal up to a constant depending only on $k$ (see Figure~\ref{g01pic}).
\begin{figure}[ht]\centering
\mbox{}\hfill\includegraphics{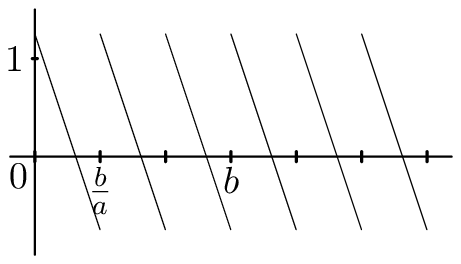}\hfill\includegraphics{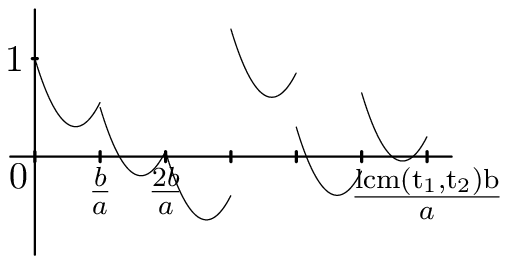}\hfill\mbox{}
\caption{Continuous $\Grat_1(P,r)$ and $\Grat_0(P,r)$}\label{g01pic}
\end{figure}
%}}}

\begin{corollary}\label{g1_corollary}%{{{
Let $T$ be as in Theorem \ref{thm_2d_triangle}, $g_i=\conv\left(
\left(\begin{matrix}0\\0\end{matrix}\right),
\left(\begin{matrix}\frac{a}{b}\frac{s_i}{t_i}\\ \frac{a}{b}\end{matrix}\right)\right)$,
for $i=1,2$.
Then $|\Grat_1(T,r)|\leq \Grat_1(T,0)$ for all $0\leq r<b\lcm{t_1,t_2}$.
More precisely, $\Grat_1(T,r)\geq -\Grat_1(T,0)+\Grat_1(g_1,r)+\Grat_1(g_2,r)$.
%}}}
\end{corollary}
\begin{proof}%{{{
It is $\Grat_2(g_i,r)=0$, $\Grat_1(g_i,r)= \frac{a}{bt_i}$, and $\Grat_0(g_i,r)=1-\left(\frac{ar}{bt_i}-\gauss{\frac{ar}{bt_i}}\right)$
for all $r\in\Q$, $i=1,2$.
Since $\frac{s_2}{t_2}-\frac{s_1}{t_1}>0$ we get
\begin{equation*}\begin{split}
\Grat_1(T,r)&=\frac{a}{b}\left(\frac{t_1+t_2}{2t_1t_2}-\left(\rat{\frac{ar}{b}}-\frac12\right)
\left(\frac{s_2}{t_2}-\frac{s_1}{t_1}\right)\right)\\
&\leq \frac{a}{b}\left(\frac{t_1+t_2}{2t_1t_2}+\frac12\left(\frac{s_2}{t_2}-\frac{s_1}{t_1}\right)\right)=\Grat_1(T,0).
\end{split}\end{equation*}
Furthermore
\begin{equation*}\begin{split}
\Grat_1(T,r)&=\frac{a}{b}\left(\frac{t_1+t_2}{2t_1t_2}-\rat{\frac{ar}{b}}
\left(\frac{s_2}{t_2}-\frac{s_1}{t_1}\right)
+\frac12\left(\frac{s_2}{t_2}-\frac{s_1}{t_1}\right)\right)\\
&\geq\frac{a}{b}\left(-\frac{t_1+t_2}{2t_1t_2}
-\frac12\left(\frac{s_2}{t_2}-\frac{s_1}{t_1}\right)\right)+\frac{a}{b}\left(\frac{t_1+t_2}{t_1t_2}\right)\\
&=-\Grat_1(T,0)+\Grat_1(g_1,r)+\Grat_1(g_2,r).\qedhere
\end{split}\end{equation*}
%}}}
\end{proof}

\begin{remark}%{{{
An analogous statement of Corollary \ref{g1_corollary} for $\Grat_0(T,\cdot)$ is not true. To see this, we
consider the triangles
\[T_\alpha=\conv\left(\binom{0}{0},\binom{\frac{\alpha-1}{\alpha}}{1},\binom{\frac{\alpha+1}{\alpha}}{1}\right),\quad \alpha\in\Z.\]
Together with Theorem \ref{thm_2d_triangle}, we get, for $m\in\Z_{\geq 0}$, $0\leq k<\alpha$ and $0\leq \tilde{r}<1$, that
\[\Grat_0(T_\alpha,m\alpha+k+\tilde{r})=\frac{1}{\alpha}\left(k(\alpha-k-2)+\tilde{r}^2-2\tilde{r}\right)+1.\]
Then for $k\sim \alpha/2$,
\[\Grat_0(T_\alpha,m\alpha+k+\tilde{r})\sim\frac{\alpha}{4}+\frac{1}{\alpha}\left(\tilde{r}^2-2\tilde{r}\right)\geq\frac{\alpha}{4}-\frac{1}{\alpha}\]
which tends to infinity, when $\alpha$ goes to infinity, but $\Grat_0(T,0)=1$.
%}}}
\end{remark}
Now we can deduce the inequality $|G_1(P,r)|\leq G_1(P,0)$ for arbitrary rational polygons:
\begin{proof}[Proof of Theorem \ref{thm:dim2}]%{{{
We first consider only integral dilation factors, that is, we show that
\begin{equation}\label{eq::g1}
|\G_1(P,k)|\leq \G_1(P,0)\quad\text{for all }k\in\Z.
\end{equation}
For this, we use several steps:\\
\emph{1. An integral version of Corollary \ref{g1_corollary} holds true for $G_1(P,k)$ for arbitrary triangles $P$ with at least one integral vertex.}
This is true since $G_1(P,k)$ is invariant under translations and unimodular transformations.\\
\emph{2. \eqref{eq::g1} is true for every rational polygon $P$ with one integral point in its interior.}
Let $P$ be an arbitrary $2$-dimensional polygon with $m$ vertices and let $z$ be an integral point in the interior
of $P$. We consider the triangulation of $P$ as given in Figure \ref{6gon}. 
\begin{figure}[ht]\centering
\includegraphics{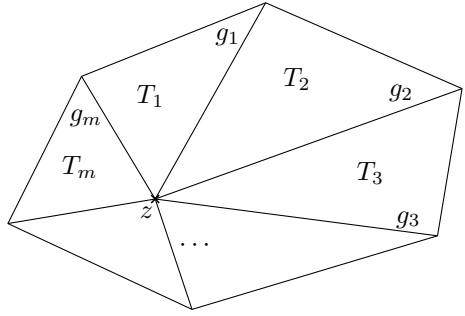}
\caption{Triangulation of $P$}\label{6gon}
\end{figure}

Here let $g_0:=g_m$. Then for all $k\in\Z$,
\begin{equation*}
\G(P,k)=\sum_{i=1}^m\G(T_i,k)-\sum_{i=1}^m\G(g_i,k)+1.
\end{equation*}
Thus, expanding all Ehrhart polynomials yields
\begin{equation*}
\G_1(P,k)
=\sum_{i=1}^m\left(\G_1(T_i,k)-\frac12\G_1(g_{i-1},k)-\frac12\G_1(g_i,k)\right).
\end{equation*}
Thus, step 1 implies that $\G_1(P,k)\leq \G_1(P,0)$, and 
\begin{equation*}\begin{split}
\G_1(P,k)&=\sum_{i=1}^m\left(\G_1(T_i,k)-\frac12\G_1(g_{i-1},k)-\frac12\G_1(g_i,k)\right)\\
&\geq\sum_{i=1}^m\left(-\G_1(T_i,0)+\frac12\G_1(g_{i-1},k)+\frac12\G_1(g_i,k)\right)\\
&= -\G_1(P,0).
\end{split}\end{equation*}
\emph{3. \eqref{eq::g1} is true for arbitrary rational polygons $P$.}
Let $l\in\Z_{\geq 0}$ such that $(l\d(P)+1)P$ contains at least one integral point in its interior.
Then for all $k\in\Z_{\geq 0}$,
\begin{equation*}
\G_1((l\d(P)+1)P,k)=\G_1(P,k)(l\d(P)+1).
\end{equation*}
Thus, from step c2 it follows that $|\G_i(P,k)|\leq \G_i(P,0)$.\\[6pt]
Finally we consider arbitrary rational dilation factors. Let $P$ be an arbitrary rational polygon and let 
$r=\frac{p}{q}\in\Q_{\geq 0}$ with $r\leq d(P)$, where $p\in\Z_{\geq  0}$ and $q\in\Z_{>0}$. Then, 
by Lemma \ref{lem:rat_dilations}, 
$\Grat_i(P,r)=\G_i\left(\frac1qP,p\right)q^i$. Hence, \eqref{eq::g1} implies that
$|\Grat_i(P,r)|\leq \Grat_i(P,0)$.
\end{proof}
%}}}
%%%%%%%%%%%%%%%%%%%%%%%%%%%%%%%%%%%%%%%%%%%%%%%%%%%%%%%%%%%%%%%%%%%%%%%%%%%%%%%%%%%%%%%%%%%%%%%%%%%%%%%%%%%%%%%

\section*{Acknowledgements}%{{{
\sloppypar
The author would like to thank Martin Henk for helpful discussions and Matthias Beck
and the reviewers for helpful comments improving the manuscript.
%}}}
%\bibliographystyle{amsplain} 
%\bibliographystyle{model1b-num-names} 
%\bibliography{ehrhart_etc}

\providecommand{\bysame}{\leavevmode\hbox to3em{\hrulefill}\thinspace}
\providecommand{\MR}{\relax\ifhmode\unskip\space\fi MR }
\providecommand{\MRhref}[2]{%
  \href{http://www.ams.org/mathscinet-getitem?mr=#1}{#2}
}
\providecommand{\href}[2]{#2}

\end{document}